\documentclass{amsart}

\usepackage{mathtools,url}

\newtheorem{theorem}{Theorem}[section]
\newtheorem{lemma}[theorem]{Lemma}
\newtheorem{proposition}[theorem]{Proposition}
\newtheorem{corollary}[theorem]{Corollary}

\theoremstyle{definition}

\newtheorem{example}[theorem]{Example}

\theoremstyle{remark}
\newtheorem{remark}[theorem]{Remark}

\numberwithin{equation}{section}

\begin{document}
	\title[Nilpotent product probability of finite rings]{Nilpotent product probability of finite rings}

	\author{Dibyasman Sarma}
	\address{Department of Mathematics, National Institute of Technology Meghalaya\\ Sohra, Meghalaya
		793108, India}
	\email{dibyasman.sarma@gmail.com}

	\author{Tikaram Subedi}
	\address{Department of Mathematics, National Institute of Technology Meghalaya\\ Sohra, Meghalaya
		793108, India}
	\email{tsubedi2010@gmail.com}
	
	
	
	
	\begin{abstract}
		For a finite ring $R$, we investigate the probability that the product of two randomly chosen elements in $R$ is nilpotent, which we call the nilpotent product probability (NPP) of $R$ and denote by $P_{nil}(R)$. We derive bounds for $P_{nil}(R)$ that depend on the structure of $R$, and prove that among all commutative rings with identity, the class of local rings with residue field $\mathbb{Z}_2$ attains the maximum value of NPP, which is equal to $3/4$. Further, we determine the set of all values of NPP's for commutative rings with identity and classify all rings attaining those values. Finally, we investigate the relationship between NPP of $R$ and NPP of some subrings of the matrix ring over $R$.
	\end{abstract}
	
	\maketitle
	
\section{Introduction}
	
In the last few decades, the study of probability in finite algebraic structures has been providing a tool for several mathematicians to investigate if the structures satisfy certain properties. Commuting probability is such an approach first studied by Gustafson \cite{gustafson} to measure how close a finite group is to being abelian. After Gustafson, MacHale \cite{machale} studied the same for finite rings and the  work has been continued by many authors, see \cite{buckley1,buckley2,buckley3,dutta}. In 2018, the idea of nullity degree of a finite ring was introduced by Esmkhani and Amiri \cite{emskhani1}, which is the probability that the product of two ring elements is zero. The same is continued in \cite{dolzan,emskhani2} and generalized to three elements in \cite{sarma}. Further, the probability that a finite commutative ring is nil-clean is studied in \cite{danchev}. Inspired by these motivational works, we come up with the idea of investigating the probability that the product of two randomly chosen ring elements is nilpotent, which we call the {\it nilpotent product probability (NPP)} of the ring.

Throughout this paper, we assume $R$ to be a finite ring, not necessarily with identity. We denote the NPP of $R$ by $P_{nil}(R)$. Thus,
\[P_{nil}(R)=\frac{|\{(x,y)\in R\times R~:~xy\in N(R)\}|}{|R|^2},\]
where $N(R)$ is the set of nilpotents in $R$, and $|S|$ is the $cardinality$ of any set $S$. $Z(R)$ is the set of $zero~divisors$ of $R$. For any integer $n\ge 2$, $\mathbb{Z}_n$ is the ring of integers modulo $n$ and $M_n(R)$ is the ring of $n$-square matrices over $R$. $T_n(R)$ and $S_n(R)$ are the subrings of $M_n(R)$, which consist of all upper triangular matrices and upper triangular matrices with scalar diagonal entries, respectively. $zp(R)$ is the nullity degree of $R$, which is $|\{(x,y)\in R\times R : xy=0\}|/|R|^2$ as defined in \cite{emskhani1}. By $\mathbb{F}_q$ we denote the finite field of size $q=p^m$ (with some prime $p$ and positive integer $m$). $R$ is called $reduced$ if $R$ has no nonzero nilpotent elements. $R$ is called $NI$ if $N(R)$ is a two sided ideal of $R$. $R$ is called $nil$ if every element of $R$ is nilpotent. $R$ is called $local$ if it has a unique maximal ideal. For any prime $p$ and positive integer $k$, by $\mathfrak{R}_{p^k}$ we denote the set of all finite commutative local rings with identity with residue field $\mathbb{F}_{p^k}$. For any two integers $a$ and $b$, $gcd(a,b)$ is the greatest common divisor of $a$ and $b$. By $\mathbb{P}$ we denote the set of prime numbers.

It is obvious that for any ring $R$, $0<P_{nil}(R)\le1$, and $P_{nil}(R)=1$ if and only if $R$ is nil. In section \ref{s1}, we study bounds for the NPP of any ring $R$, which depend on the structure of $R$. This section is further divided into three subsections. In section \ref{ss1}, we investigate the NPP for NI rings. In section \ref{ss2}, we explore the NPP for commutative rings with identity providing a complete classification of such rings using this probability measure. In section \ref{ss3}, we establish the relationships between the NPP's of $T_n(R)$ and $R$, as well as the NPP's of $S_n(R)$ and $R$. Moreover, we acquire the maximum values of the same for both these classes of matrix rings over any commutative ring with identity.

\medspace

\section{Bounds for $P_{nil}(R)$}\label{s1}

For any $x\in R$, we define the set $n(x)$ as 
\[n(x):=\{y\in R : xy\in N(R)\}.\]

The following lemma is straightforward.

\begin{lemma}\label{l4}
	For any ring $R$, we have \[P_{nil}(R)=\frac{\sum\limits_{x\in R}|n(x)|}{|R|^2}.\]
\end{lemma}

We obtain the following result, where we prove the multiplicative property of the NPP with respect to direct products of rings.

\begin{proposition} \label{l3}
	For any two rings $R_1$ and $R_2$, we have 
	\[P_{nil}(R_1\times R_2)=P_{nil}(R_1)P_{nil}(R_2).\]
\end{proposition}
\begin{proof}
	For any $(x_1,x_2)\in R_1\times R_2$, we have
	\begin{align*}
		n\big((x_1,x_2)\big)&=\big\{(y_1,y_2)\in R_1\times R_2 : (x_1,x_2)(y_1,y_2)\in N(R_1\times R_2)\big\}\\
		&=\big\{(y_1,y_2)\in R_1\times R_2 : x_1y_1\in N(R_1),x_2y_2\in N(R_2)\big\}\\
		&=\big\{y_1\in R_1 : x_1y_1\in N(R_1)\big\}\times \big\{y_2\in R_2 : x_2y_2\in N(R_2)\big\}\\
		&=n(x_1)\times n(x_2).
	\end{align*}
	Thus,
	\begin{align*}
		P_{nil}(R_1\times R_2)&=\frac{\sum\limits_{(x_1,x_2)\in R_1\times R_2}\big|n\big((x_1,x_2)\big)\big|}{|R_1\times R_2|^2}\\
		&=\frac{\sum\limits_{x_1\in R_1}|n(x_1)|\sum\limits_{x_2\in R_2}|n(x_2)|}{|R_1|^2|R_2|^2}\\
		&=P_{nil}(R_1)P_{nil}(R_2).\qedhere
	\end{align*}
\end{proof}

\medspace

In the following, for any ring $R$, we establish a comparison between the nullity degree and the NPP of $R$.

\begin{proposition}\label{p2}
	Let $R$ be a ring. Then
	\[zp(R)\le P_{nil}(R).\]
	Moreover, the equality holds if and only if any nonzero nilpotent element cannot be expressed as a product of two elements.
\end{proposition}

\begin{remark}\label{r1}
	Note that, if $R$ is reduced, the property that ``any nonzero nilpotent element cannot be expressed as a product of two elements" holds trivially in $R$. However, there exists a non reduced ring with this certain property, for example, the ring $2\mathbb{Z}_{12}$. Further, observe that the nullity degree and NPP of $2\mathbb{Z}_{12}$ are equal to $4/9$.
\end{remark}

\begin{corollary}\label{c6}
	Let $R$ be a ring of order $n$. Then 
	\[P_{nil}(R)\ge\frac{2n-1}{n^2}.\]
	The equality holds if and only if either $R\cong \mathbb{F}_n$ or $R=\{0\}$.
\end{corollary}
\begin{proof}
	It follows from Proposition \ref{p2} and \cite[Lemma 3.1]{emskhani1}.
\end{proof}

\medspace

\subsection{NPP for NI rings.}\label{ss1}
Since $N(R)$ need not be an ideal of $R$ in general, throughout this subsection, $R$ is assumed to be NI. Note that, for any NI ring $R$ and $x\in R$, $n(x)$ is a right ideal of $R$ and $N(R)\subseteq n(x)$. In contrast, when $R$ is not NI, these properties may fail. For example, consider the ring $M_2(\mathbb{Z}_2)$ and the elements $I=\begin{bmatrix}
	1 & 0\\0 & 1
\end{bmatrix}, ~E_{12}=\begin{bmatrix}
	0 & 1\\0 & 0
\end{bmatrix}, ~E_{21}=\begin{bmatrix}
0 & 0\\1 & 0
\end{bmatrix}$, where $n(I)$ is not an ideal as $E_{12}, E_{21}\in n(I)$, but $E_{12}+E_{21}\notin n(I)$, and $N(M_2(\mathbb{Z}_2))\not\subseteq n(E_{21})$ as $E_{12}\in N(M_2(\mathbb{Z}_2))$, but $E_{12}\notin n(E_{21})$.

\medskip

We provide the following lemma, which will be useful.

\begin{lemma}[{\cite{Kobayashi}}]\label{l5}
	Let $R$ be a ring. Then either $R$ has identity or every element is a zero divisor. In fact, if $R$ is with identity, then every element is either a unit or a zero divisor.
\end{lemma}

In the following, we establish the NPP for any ring $R$.

\begin{theorem} \label{l2}
	For any ring $R$, we have
		\[P_{nil}(R)=\displaystyle\frac{(2|R|-|Z(R)|)|N(R)|+\sum\limits_{x\in Z(R)\setminus N(R)}|n(x)|}{|R|^2}.\]
\end{theorem}
\begin{proof}
	From Lemma \ref{l4}, we have
	\[P_{nil}(R)=\displaystyle\frac{\sum\limits_{x\in N(R)}|n(x)|+\sum\limits_{x\in Z(R)\setminus N(R)}|n(x)|+\sum\limits_{x\not\in Z(R)}|n(x)|}{|R|^2}.\]
	It is immediate that for $x\in N(R)$, $n(x)=R$. Now, we have two cases:\\
	\textbf{Case 1.} Suppose that $R$ is with identity.  Let $x\not\in Z(R)$. Then $x$ is a unit. Assume that $y\in n(x)$. Then $xy\in N(R)$, and so, $x^{-1}(xy)\in N(R)$, or, $y\in N(R)$. Therefore, $n(x)=N(R)$. Thus,
	\[P_{nil}(R)=\frac{|R||N(R)|+(|R|-|Z(R)|)|N(R)|+\sum\limits_{x\in Z(R)\setminus N(R)}|n(x)|}{|R|^2}.\]
	\textbf{Case 2.} Suppose that $R$ is without identity. Then we have $Z(R)=R$. Thus, 
	\[P_{nil}(R)=\frac{|R||N(R)|+\sum\limits_{x\in R\setminus N(R)}|n(x)|}{|R|^2}.\]
	 Hence, combining both the cases, we have
	 \[P_{nil}(R)=\frac{(2|R|-|Z(R)|)|N(R)|+\sum\limits_{x\in Z(R)\setminus N(R)}|n(x)|}{|R|^2}.\qedhere\] 
\end{proof}

\medspace

In the following result, we derive bounds for $P_{nil}(R)$ in terms of the size of $R, Z(R)$ and $N(R)$. Further, we discuss the condition under which the bounds become sharp.
 
\begin{theorem} \label{t1}
	For any ring $R$, we have
	\begin{enumerate}
		\item[{(1)}] $P_{nil}(R)\ge \displaystyle\frac{(2|R|-|Z(R)|)|N(R)|}{|R|^2}$ and 
		
		\item[{(2)}] $P_{nil}(R)\le\displaystyle\frac{2(|R|-|Z(R)|)|N(R)|+|Z(R)|^2}{|R|^2}$.
	\end{enumerate}
	Moreover, the equality holds both in (1) and (2) if and only if $N(R)=Z(R)$.
\end{theorem}
\begin{proof}
	The lower bound follows directly from Theorem \ref{l2}. The upper bound follows since for any $x\in Z(R)\setminus N(R)$, we have $n(x)\subseteq Z(R)$. Now, notice that the lower and upper bounds coincide and equal to $(2|R|-|N(R)|)|N(R)|/|R|^2$ if and only if $n(x)=Z(R) ~\forall x\in Z(R)\setminus N(R)$, equivalently, $xy\in N(R) ~\forall x,y\in Z(R)$, which  occurs if and only if $N(R)=Z(R)$. Hence, the required result holds.
\end{proof}

\medspace

The following corollary is a consequence of Theorem \ref{t1} for a ring $R$ without identity, which is not nil.

\begin{corollary}\label{c4}
	Let $R$ be a ring without identity. If $R$ is not nil, then
	\[P_{nil}(R)>\frac{(2|R|-|N(R)|)|N(R)|}{|R|^2}.\]
\end{corollary}

\begin{example}
	(1) For the ring $T_2(\mathbb{Z}_2)$, $Z(T_2(\mathbb{Z}_2))=\left\{\begin{bmatrix}
			0 & 0\\0 & 0
		\end{bmatrix},\begin{bmatrix}
		1 & 0\\0 & 0
		\end{bmatrix},\begin{bmatrix}
		0 & 1\\0 & 0
		\end{bmatrix},\right.$ $\left.\begin{bmatrix}
		0 & 0\\0 & 1
		\end{bmatrix},\begin{bmatrix}
		1 & 1\\0 & 0
		\end{bmatrix},\begin{bmatrix}
		0 & 1\\0 & 1
		\end{bmatrix}\right\}$ and $N(T_2(\mathbb{Z}_2))=\left\{\begin{bmatrix}
		0 & 0\\0 & 0
		\end{bmatrix},\begin{bmatrix}
		0 & 1\\0 & 0
		\end{bmatrix}\right\}$. Therefore, $1/8\le P_{nil}(T_2(\mathbb{Z}_2))\le 11/16$. Observe that a simple calculation yields that $P_{nil}(T_2(\mathbb{Z}_2))=9/16$.

        \medskip
                
		\noindent(2) For the ring $\mathbb{Z}_8$, $Z(\mathbb{Z}_8)=N(\mathbb{Z}_8)=\{\overline{0},\overline{2},\overline{4},\overline{6}\}$. So, $P_{nil}(\mathbb{Z}_8)=3/4=\{(2.8-4)4\}/8^2$.
	
\end{example}

\medspace

\subsection{NPP for commutative rings.}\label{ss2}
In this subsection, we investigate the NPP for finite commutative rings with identity. Throughout, we shall use \cite[Theorem~8.7 ]{atiyah} ({\it structure theorem for finite commutative rings}), which says that any finite commutative ring with identity is unique (up to isomorphism) direct product of local rings. We begin by stating the following lemma, which will be used subsequently.

\begin{lemma}\label{l6}
	Let $R$ be a local ring. Then $N(R)$ becomes the unique maximal ideal of $R$. Moreover, we have  $|R|=p^\alpha$  for some prime $p$ and integer $\alpha>0$.
\end{lemma}
\begin{proof}
See \cite[Theorem 2]{ragha}.
\end{proof}

\medspace

Now, before we establish the NPP for local rings, we recall the class of rings $\mathfrak{R}_{p^k}$ as mentioned in the introduction. Note that, there exists at least one ring which belongs to $\mathfrak{R}_{p^k}$ for any prime $p$ and integer $k>0$, for example, $\mathbb{F}_{p^k}$. Again, if $R$ is local, then $R\in\mathfrak{R}_{p^k}$ for some prime $p$ and integer $k>0$.

\begin{theorem} \label{t2}
	Let $R$ be a ring. Then
	\begin{enumerate}
		\item[{(1)}]  $R$ is local if and only if
		\[P_{nil}(R)=\frac{(2|R| -|N(R)|)|N(R)|}{|R|^{2}}.\]
		
		\item[{(2)}] Furthermore, if $R$ is local and $p, \alpha$ are prime and positive integer, respectively, then $P_{nil}(R)=\frac{2p^k-1}{p^{2k}}$ if and only if $R\in\mathfrak{R}_{p^k}$.
	\end{enumerate}
\end{theorem}
\begin{proof}
	(1) Since a commutative ring $R$ with identity is local if and only if every zero divisor is nilpotent, therefore, from Theorem \ref{t1}, we have
		\[P_{nil}(R)= \frac{(2|R|-|N(R)|)|N(R)|}{|R|^2}\]
	if and only if $R$ is local.
	
	\medspace
	
	(2) From (1), we have 
	$P_{nil}(R)= \frac{2|R/N(R)|-1}{|R/N(R)|^2}$
	if and only if $R$ is local. Now, if $R\in\mathfrak{R}_{p^k}$, then $|R/N(R)|=p^k$, whence the result holds. Conversely, let $P_{nil}(R)=\frac{2p^k-1}{p^{2k}}$. Since $R$ is local, therefore, $R\in\mathfrak{R}_{q^s}$ for some prime $q$ and integer $s>0$. Thus,
	$P_{nil}(R)=\frac{2q^s-1}{q^{2s}}=\frac{2p^k-1}{p^{2k}},$
	which implies $q^{2s}\big|p^{2k}(2q^{s} -1)$. But $gcd(q^{2s},2q^{s} -1)$ $=1$. So, $q^{2s} ~\big|~p^{2k}$, and therefore, $q=p$ and $s\le k$. Similarly, we can prove $k\le s$, and hence, $k=s$. 
\end{proof}

\medspace

This leads to the following corollaries.

\begin{corollary} \label{o2}
	Let $\mathcal{S}_{local}$ be the set of all values of $P_{nil}(R)$ for all local rings $R$. Then
	\[\mathcal{S}_{local}=\left\{\frac{2p^k-1}{p^{2k}} : p\in\mathbb{P}, k>0 \right\}.\]
\end{corollary}

\begin{corollary} \label{c2}
	Let $R$ be a local ring such that $|R|=p^\alpha$, where $p$ is a prime and $\alpha>0$ is an integer. Then
	\[ P_{nil}(R)\le \frac{2p-1}{p^2}.\]
	The equality holds if and only if $R\in\mathfrak{R}_p$.
\end{corollary}
\begin{proof}
	From Theorem \ref{t2}(1), we have 
	\[P_{nil}(R)=\frac{(2p^\alpha -p^i)p^i}{p^{2\alpha}},\]
	where $|N(R)|=p^i$ for some $0\le i\le \alpha-1$. Since for integers $0\le i<j\le \alpha -1$, we get $(2p^\alpha-p^i)p^i<(2p^\alpha-p^j)p^j$, therefore,  the required inequality follows. Further, the equality holds from Theorem \ref{t2}(2).
\end{proof}

\medspace

The following result is a consequence of Proposition \ref{l3} and Theorem \ref{t2}.

\begin{theorem} \label{t3}
	Let $R$ be a ring such that $R=R_1\times R_2\times ... \times R_t,~t\ge 1$, where $R_i\in\mathfrak{R}_{p_{i}^{k_i}}$, $p_i\in\mathbb{P},~k_i> 0$ for $1\le i\le t$. Then
	\[
	P_{nil}(R)=\prod\limits_{i=1}^{t}\frac{2p^{k_i}_i-1}{p^{2k_i}_i}.
	\]
\end{theorem}

\begin{example} \label{e2}
	We know that for any prime $p$ and positive integer $\alpha$,  $\mathbb{Z}_{p^\alpha}\in\mathfrak{R}_p$, and therefore, $P_{nil}(\mathbb{Z}_{p^\alpha})=\frac{2p-1}{p^2}$. In particular, if $n=\prod\limits_{i=1}^{t}p_i^{\alpha_i}$,  where $p_i'$s are distinct primes and $\alpha_i'$s are positive integers, then $\mathbb{Z}_n\cong\mathbb{Z}_{p_1^{\alpha_1}}\times \mathbb{Z}_{p_2^{\alpha_2}}\times...\times \mathbb{Z}_{p_t^{\alpha_t}}$, and hence, $P_{nil}(\mathbb{Z}_n)=\prod\limits_{i=1}^{t}\frac{2p_i-1}{p^2_i}.$
\end{example}

\begin{remark}\label{r2}
	Note that, Theorem \ref{t2}(2) need not hold for non-local rings. For example, let $R=(\mathbb{Z}_2)^3 \times \mathbb{Z}_3$. Then we have  $P_{nil}(R)=15/64=(2\cdot8-1)/8^2$, but $R\notin \mathfrak{R}_8$.
\end{remark}

\medspace

Now, we obtain the following corollaries that follow from the preceding results. 

\begin{corollary} \label{o3}
	Let $\mathcal{S}$ be the set of all values of $P_{nil}(R)$ for all rings $R$. Then
	 \[\mathcal{S}=\left\{\displaystyle\prod_{i=1}^{t}\frac{2p_i^{k_i}-1}{p_i^{2k_i}} : p_i\in\mathbb{P};t,k_i>0\right\}.\]
\end{corollary}

\begin{corollary}\label{c11}
	Let $R$ be a ring such that $|R|=p^\alpha$ for some prime $p$ and integer $\alpha>0$. Then
	\begin{enumerate}
		\item[{(1)}] $\displaystyle P_{nil}(R)\ge\frac{2p^\alpha-1}{p^{2\alpha}}$, with the equality if and only if $R$ is a field.
		\item[{(2)}] $\displaystyle P_{nil}(R)\le \frac{2p-1}{p^2}$, with the equality if and only if $R\in\mathfrak{R}_p$.
	\end{enumerate}
\end{corollary}
\begin{proof}
	(1) This follows directly from Corollary \ref{c6}.
	
	(2) For local $R$, it is already proved in Corollary \ref{c2}. Suppose that $R$ is non-local. Then $R\cong R_1\times R_2\times ...\times R_t, t\ge 2$, where each $R_i$ is local with $P_{nil}(R_i)\le \frac{2p-1}{p^2}$, which implies $P_{nil}(R)<\frac{2p-1}{p^2}$.
\end{proof}

\begin{corollary} \label{c3}
	Let $R$ be a ring such that $|R|=\prod\limits_{i=1}^{t}p_i^{\alpha_i}$, where $p_i'$s are distinct primes and $\alpha_i'$s are positive integers. Then
	\begin{enumerate}
		\item[{(1)}] $P_{nil}(R)\ge\prod\limits_{i=1}^{t}\frac{2p_i^{\alpha_i}-1}{p_i^{2\alpha_i}}$, with the equality if and only if $R\cong \mathbb{F}_{p_1^{\alpha_1}}\times \mathbb{F}_{p_2^{\alpha_2}}\times ...\times \mathbb{F}_{p_t^{\alpha_t}}$.
		
		\item[{(2)}] $P_{nil}(R)\le\prod\limits_{i=1}^{t}\frac{2p_i-1}{p_i^2}$, with the equality if and only if $R\cong R_1\times R_2\times ...\times R_t$, where each $R_i\in\mathfrak{R}_{p_i}$.
	\end{enumerate}
\end{corollary}
\begin{proof}
	Since $R\cong R_1\times R_2\times ... \times R_t$, where each $R_i$ is a ring of order $p_i^{\alpha_i}$, therefore, by Corollary \ref{c11} the required result holds.
\end{proof}

\begin{corollary}\label{c5}
	Let $R$ be a ring and $p$ be the largest prime that divides the order of $R$. Then
	\[P_{nil}(R)\le \frac{2p-1}{p^2}.\]
\end{corollary}

\medspace

In the following, we determine the class of rings, which attains the maximum value of $P_{nil}(R)$.

\begin{theorem} \label{t6}	
	Let $R$ be a ring. Then
	\begin{enumerate}
		\item[{(1)}] For any prime $p$, $\displaystyle P_{nil}(R)=\frac{2p-1}{p^2}$ if and only if $R\in \mathfrak{R}_p$.
		
		\item[{(2)}] In particular, $P_{nil}(R)\le 3/4$, and the equality holds if and only if $R\in \mathfrak{R}_2$.
	\end{enumerate} 
\end{theorem}
\begin{proof}
	(1) Suppose that $R\in \mathfrak{R}_p$. Then from Theorem \ref{t2}, $P_{nil}(R)=\frac{2p-1}{p^2}$. Conversely, suppose that $P_{nil}(R)=\frac{2p-1}{p^2}$. Assume that $R$ is non-local. Then $R\cong R_1\times R_2\times ... \times R_t,~t\ge 2$, where each $R_i\in\mathfrak{R}_{p_{i}^{k_i}}$ for $p_i\in\mathbb{P},~k_i> 0$. Then
	\begin{align*}
		&P_{nil}(R)=\prod\limits_{i=1}^{t}\frac{2p^{k_i}_i-1}{p^{2k_i}_i}=\frac{2p-1}{p^2}\\
		\implies&p^2\Big|(2p-1)\prod\limits_{i=1}^{t}p^{2k_i}_i\\
		\implies&p^2\Big|\prod\limits_{i=1}^{t}p^{2k_i}_i \text{ as }gcd(p^2, 2p-1)=1\\
		\implies&p\Big|\prod\limits_{i=1}^{t}p^{k_i}_i.
	\end{align*}
	So, $p\in\{p_1, p_2, ..., p_t\}$. Let $p=p_j$ for some $j$, $1\le j\le t$. Then $P_{nil}(R_j)\le\frac{2p-1}{p^2}$, which implies $P_{nil}(R)< \frac{2p-1}{p^2}$, a contradiction. Therefore, $R$ is local whose order is divisible by $p$, and hence, $R\in\mathfrak{R}_p$.
	
	\medspace
	
	(2) Assume that $R$ is local. Then $|R|=p^\alpha$ for some prime $p$ and integer $\alpha>0$. So, $P_{nil}(R)\le\frac{2p-1}{p^2}\le 3/4$. Also, the equality holds if and only if $R\in\mathfrak{R}_2$. Next, assume that $R$ is non-local. Then $R\cong R_1\times R_2\times ...\times R_t$ for some $t\ge 2$, where each $R_i$ is local with $P_{nil}(R_i)\le 3/4$. Therefore, $P_{nil}(R)< 3/4$.
\end{proof}

\medspace

Based on the previously discussed results, we provide the following table, which illustrates the possible values of NPP's and the class of rings attaining those values. In this table, we provide all class of rings with NPP at least $15/64$. It can be noted that the table can be extended for NPP's less than 15/64 as well providing the complete classification. Throughout this table, $\mathfrak{R}_{q_1}\times \mathfrak{R}_{q_2}$ denotes the class of rings $\{R_1\times R_2 : R_1\in\mathfrak{R}_{q_1}, R_2\in\mathfrak{R}_{q_2}\}$, where $q_1=p_1^{m_1}, q_2=p_2^{m_2}$ for primes $p_1, p_2$ and positive integers $m_1, m_2$.

\begin{table}[h!]
	\centering
	\caption{Rings $R$ with $P_{nil}(R)\ge 15/64$}\label{tab1}
	\setlength{\tabcolsep}{2.6pt}
	\setlength{\arrayrulewidth}{1pt}
	\renewcommand{\arraystretch}{1.4}
	\begin{tabular}{|c|c|c|}
		\hline
		Sl. No. & $P_{nil}(R)$	& All possible class of rings	\\ 
		\hline
		$1$      & $\frac{3}{4}=0.75$ 		  	&$\mathfrak{R}_2$\\ 
		$2$      &  	$\frac{9}{16}=0.5625$	  	&$\mathfrak{R}_2^2$\\ 
		$3$      &	$\frac{5}{9}\approx 0.5556$	    &$\mathfrak{R}_3$\\
		$4$      &    $\frac{7}{16}=0.4375$ 		&$\mathfrak{R}_4$\\ 
		$5$      &    $\frac{27}{64}\approx 0.4219$  	&$\mathfrak{R}_2^3$\\ 
		$6$      &    $\frac{15}{36}\approx 0.4167$     &$\mathfrak{R}_2 \times \mathfrak{R}_3$\\ 
		$7$      &	$\frac{9}{25}=0.36$		&$\mathfrak{R}_5$\\ 
		$8$      &    $\frac{21}{64}\approx 0.3281$     &$\mathfrak{R}_2 \times \mathfrak{R}_4$\\ 	
		$9$      &    $\frac{81}{256}\approx 0.3164$    &$\mathfrak{R}_2^4$\\ 
		$10$     &    $\frac{5}{16}=0.3125$     &$\mathfrak{R}_2^2\times\mathfrak{R}_3$\\ 
		$11$     &    $\frac{25}{81}\approx 0.3086$     &$\mathfrak{R}_3^2$\\ 
		$12$     &   	$\frac{27}{100}=0.27$	&$\mathfrak{R}_2 \times \mathfrak{R}_5$\\
		$13$     &   	$\frac{13}{49}\approx 0.2653$	&$\mathfrak{R}_7$\\ 
		$14$     &   	$\frac{63}{256}\approx 0.2461$	&$\mathfrak{R}_2^2 \times \mathfrak{R}_4$\\ 
		$15$     &   	$\frac{35}{144}\approx 0.2431$	&$\mathfrak{R}_3 \times \mathfrak{R}_4$\\ 
		$16$     &   	$\frac{243}{1024}\approx 0.2373$	&$\mathfrak{R}_2^5$\\ 
		$17$     &   	$\frac{15}{64}\approx 0.2344$		&$\mathfrak{R}_8, ~ \mathfrak{R}_2^3 \times \mathfrak{R}_3$\\ 
		\hline
	\end{tabular}
\end{table}

\medspace

\subsection{NPP for the ring of triangular matrices.}\label{ss3}
In this subsection, we investigate the NPP for the rings $T_n(R)$ and $S_n(R)$ for any ring $R$ and integer $n\ge 2$. The following lemma will be required to prove our main result.

\begin{lemma}\label{l1}
	For any matrix $[a_{ij}]\in T_n(R), 1\le i,j\le n$, $[a_{ij}]\in N(T_n(R))$ if and only if $a_{ii}\in N(R)~ \forall ~1\le i\le n$.
\end{lemma}

\begin{theorem}\label{t5}
	Let $R$ be a ring and $n\ge 2$ be an integer. Then 
	\begin{enumerate}
		\item[{(1)}] $P_{nil}(T_n(R))=[P_{nil}(R)]^n$.
		
		\item[{(2)}]  $P_{nil}(S_n(R))=P_{nil}(R)$.
	\end{enumerate}
\end{theorem}
\begin{proof}
	(1) Suppose that $[a_{ij}], [b_{ij}]\in T_n(R)$. Then by Lemma \ref{l1}, $[a_{ij}][b_{ij}]\in N(T_n(R))$ if and only if $a_{ii}b_{ii}\in N(R) ~\forall ~1\le i\le n$. Thus, the probability that $[a_{ij}][b_{ij}]$ is nilpotent in $T_n(R)$ is equal to the probability that $a_{ii}b_{ii}$ is nilpotent in $R$ $\forall ~1\le i\le n$. Hence, $P_{nil}(T_n(R))=[P_{nil}(R)]^n$.
	
	\medspace
	
	(2) It follows analogously.
\end{proof}

\medspace

As a result of this, we obtain the following corollaries.

\begin{corollary}\label{c1}
	For any ring $R$ and integer $n\ge2$, we have
	\begin{enumerate}
		\item[{(1)}] $P_{nil}(T_n(R))\ge\left(\frac{2|R|-1}{|R|^2}\right)^n$, with the equality if and only if either $R$ is a field or $R=\{0\}$.
		
		\item[{(2)}]  $P_{nil}(S_n(R))\ge\frac{2|R|-1}{|R|^2}$, with the equality if and only if either $R$ is a field or $R=\{0\}$.
	\end{enumerate}
\end{corollary}

\begin{corollary}\label{c9}
	Let $R$ be a ring. Then
	\begin{enumerate}
		\item[{(1)}] For integers $m>n\ge2$, $P_{nil}(T_n(R))>P_{nil}(T_m(R))$.
		
		\item[{(2)}] For all integers $n,m\ge2$, $P_{nil}(S_n(R))=P_{nil}(S_m(R))$.
	\end{enumerate}
\end{corollary}

\begin{corollary}\label{c10}
	For any two rings $R_1,R_2$ and integer $n\ge2$, the following are equivalent:
	\begin{enumerate}
		\item[{(1)}] $P_{nil}(R_1)= P_{nil}(R_2)$.
		
		\item[{(2)}] $P_{nil}(T_n(R_1))= P_{nil}(T_n(R_2))$.
		
		\item[{(3)}] $P_{nil}(S_n(R_1))= P_{nil}(S_n(R_2))$.
	\end{enumerate}
\end{corollary}

\medspace

In the following corollary, for any commutative ring $R$ with identity, we obtain the maximum values of the NPP's for the rings $T_n(R)$ and $S_n(R)$.

\begin{corollary}\label{c8}
	For any commutative ring $R$ with identity and integer $n\ge2$, we have
	\begin{enumerate}
		\item[{(1)}] $P_{nil}(T_n(R))\le 9/16$, and the equality holds if and only if $n=2$ and $R\in\mathfrak{R}_2$.
		
		\item[{(2)}]  $P_{nil}(S_n(R))\le 3/4$, and the equality holds if and only if  $R\in\mathfrak{R}_2$.
	\end{enumerate}
\end{corollary}
\begin{proof}
	(1) We have $P_{nil}(T_n(R))=[P_{nil}(R)]^n$. Now, using Theorem \ref{t6}(2), we get 	$P_{nil}(T_n(R))\le (3/4)^n\le (3/4)^2=9/16$, and the equality holds if and only if $n=2$ and $R\in\mathfrak{R}_2$.
	
	\medspace
	
	(2) Similarly, we have $P_{nil}(S_n(R))=P_{nil}(R)\le 3/4$, with the equality if and only if $R\in\mathfrak{R}_2$.
\end{proof}

\medskip

\section*{Conclusion}
In this paper, we have addressed a probabilistic aspect on finite rings, which provides a numerical measurement on the nilpotency behaviour of the product of two randomly chosen elements. Additionally, we have characterized all commutative rings with identity using this probability measure, which aids in knowing the structure of those rings. The classification of arbitrary rings via this probability measure remains a prospective avenue for future research.

\end{document}